\documentclass[10pt,twoside,reqno,centertags]{amsart}
\usepackage{amsmath,amsthm,amsfonts,amssymb}
\usepackage{graphicx}
\usepackage{color}
\providecommand{\U}[1]{\protect\rule{.1in}{.1in}}
\newtheorem{theorem}{Theorem}[section]
\newtheorem{rem}{Remark}

\newtheorem{lemma}[theorem]{Lemma}

\newtheorem{proposition}[theorem]{Proposition}

\numberwithin{equation}{section}

\newcommand{\R}{\mathbb{R}}
\newcommand{\Rn}{\mathbb{R}^n}
\newcommand{\Rnp}{\mathbb{R}^n_+}
\newcommand{\be}{\begin{equation} \label}
\newcommand{\ee}{\end{equation}}
\newcommand{\eps}{\varepsilon}

\begin{document}

\title[Lane-Emden system in the half-space]{A Liouville theorem for the \\ Lane-Emden system in the half-space}

\author[Yimei Li and Philippe Souplet]{Yimei Li $^{(1, 2)}$ and Philippe Souplet $^{(2)}$}  

\thanks{$^{(1)}$School of Mathematics and Statistics, Beijing Jiaotong University, Beijing 100044, PR China.}

\thanks{$^{(2)}$Universit\'e Sorbonne Paris Nord, CNRS UMR 7539, LAGA, 93430 Villetaneuse, France.}

\thanks{Emails: liyimei@bjtu.edu.cn, souplet@math.univ-paris13.fr}

\begin{abstract}
We prove that the Dirichlet problem for the Lane-Emden system in a half-space has no positive classical solution that is bounded on finite strips. 
Such a nonexistence result was previously available only for bounded solutions or under a restriction on the powers in the nonlinearities.
\end{abstract}

\maketitle

\section{Introduction}

\subsection{Background}

We are interested in the nonexistence question for the Lane-Emden system
\be{LEDsys}
\begin{cases}
\begin{aligned}
-\Delta u&=v^{p}&\quad&\mbox{\rm in }\  \mathbb{R}_+^{n},\\
-\Delta v&=u^{q}&\quad&\mbox{\rm in }\  \mathbb{R}_+^{n},\\
u=v&=0&\quad&\mbox{\rm on }\  \partial\mathbb{R}_+^{n}
\end{aligned}
\end{cases}
\ee
where $\mathbb{R}_+^{n}=\{x\in \mathbb{R}^{n}:x_n>0\}$ is the half-space and $p,q>1$.
In this article, by a solution we always mean a positive classical solution, unless otherwise mentioned.

\vskip 1mm

To motivate our result, let us give some background. For this we go back to the celebrated Liouville  theorem of Gidas and Spruck  \cite{GS}, which states that the Lane-Emden equation 
\be{LEeq}
-\Delta u=u^{p}\quad\mbox{\rm in }\  \mathbb{R}^{n}
\ee
does not possess any solution  provided $1<p<p_S:=(n+2)/(n-2)_+$ 
(see also \cite{BV,CL,LZ} 
for other proofs). If $p\geq p_S$, there exist radial, bounded solutions. The natural counterpart of the  Lane-Emden equation in elliptic systems is the Lane-Emden system
\be{LEsys}
\begin{cases}
\begin{aligned}
&-\Delta u=v^{p}&\quad&\mbox{\rm in }\  \mathbb{R}^{n},\\
&-\Delta v=u^{q}&\quad&\mbox{\rm in }\  \mathbb{R}^{n},
\end{aligned}
\end{cases}
\ee
with $p,q>1$, which has also received considerable attention, 
being in particular a model case of Hamiltonian system.
Instead of the Sobolev exponent $p_S$, the critical role is played by the so-called Sobolev hyperbola
(cf.~\cite{CFM,PVV,Mit93}).
Indeed, if 
$$\frac{1}{p+1}+\frac{1}{q+1}\le {n-2\over n}, $$
then system \eqref{LEsys} admits some radial, bounded solution (see~\cite{SZ98})
whereas, if 
\be{hypLE3}
{1\over p+1}+{1\over q+1}>{n-2\over n}
\ee
and $n\le 4$, then system \eqref{LEsys} does not possess any solution.
The latter was first proved in  \cite{SZ96} in dimension $n=3$ for polynomially bounded solutions,
and this additional growth restriction was later removed in \cite{PQS}. 
The result for $n=4$ was then proved in \cite{So09}.
It is conjectured that the nonexistence holds under assumption \eqref{hypLE3} in any dimension
but only partial results are available in dimensions $n\ge 5$.
For instance, it follows from \cite{deF,RZ} that \eqref{LEsys} does not admit any 
solution if $p,q<p_S$; see \cite{CP,Mit,BM,Lin,MP,So09} for additional results. 

\vskip 1mm

On the other hand, the research for problems \eqref{LEeq} and \eqref{LEsys}  
on proper unbounded subdomains of $\mathbb{R}^{n}$ also has a long history and a large number of studies.
Consider the Dirichlet problem
\be{LEDeq}
\begin{cases}
\begin{aligned}
-\Delta u&=u^{p}&\quad&\mbox{\rm in }\  \mathbb{R}_{+}^{n},\\
u&=0&\quad&\mbox{\rm on }\  \partial\mathbb{R}_{+}^{n},
\end{aligned}
\end{cases}
\ee
in the half-space $\mathbb{R}_+^{n}$, with $p>1$. 
Recall that Liouville type theorems in the half-space are important both for their intrinsic interest
(see,~e.g.,~\cite{BCN1,BCN2,CW,HLSWW}, including connections with the De~Giorgi conjecture)
and for their applications to a priori estimates by rescaling methods (cf.~\cite{GS2} 
and see \cite{QS19} and the references therein).
It was first proved in \cite{GS2}, by moving planes arguments, that \eqref{LEDeq} does not possess solutions for $p\le p_S$.
For bounded solutions, this condition was improved in \cite{Dancer} to $p<p'_S:=(n+1)/(n-3)_{+}$,
observing that such a solution is monotone in the $x_n$ direction (as a consequence of further moving plane arguments)
and gives rise, as $x_n\to\infty$, to a solution of \eqref{LEeq} in $\mathbb{R}^{n-1}$ which cannot exist in view 
of the Gidas-Spruck result. 
A further important development was made in \cite{Farina}, where variational estimates and stability properties were used to show that \eqref{LEDeq}  does not possess bounded solutions for $n\le 11$ 
or $p<p_{JL}(n-1):=1+4\frac{n-5+2\sqrt{n-2}}{(n-3)(n-11)}$.  
Another breakthrough was then made in~\cite{CLZ}, where nonexistence of bounded solutions was showed for any $p>1$,
thus providing a complete answer to the question in the class of bounded solutions.
However, after this, the best result for unbounded solutions remained that of \cite{GS2}, limited to the range $p\le p_S$.
The situation was recently improved in \cite{DSS}, by showing that for any $p>1$,
there are no solutions that are monotone in the $x_n$ direction,
nor solutions that are bounded on finite strips, and then in \cite{DFT}, where nonexistence was proved
for solutions that are stable outside of a compact.

\vskip 1mm

Let us turn to the Lane-Emden system in the half-space \eqref{LEDsys}. 
Extending the arguments of \cite{Dancer}, it was proved in \cite{BiMi} that for given $p, q>1$, if system \eqref{LEsys} does not admit any bounded solution, then \eqref{LEDsys} does not admit any bounded solution in dimension $n+1$.
As for the nonexistence of (possibly unbounded) solution of \eqref{LEDsys},
it was established in \cite{RZ} for $p,q<p_S$ (by moving spheres arguments),
and next in \cite{PQS} under the assumption that system \eqref{LEDsys} with same $p, q$ does not admit any bounded solution. Note that, by \cite{BiMi, PQS}, the nonexistence results 
for \eqref{LEsys} in the previous paragraph have direct implications for \eqref{LEDsys}.
It was then shown in \cite{Cowan}, by variational estimates and stability arguments, that there is no bounded solution if 
$p,q\ge 2$ and $n\le 11$ (or under some upper restrictions on $p,q$ in higher dimensions).
Finally in \cite{CLZ}, as well as for the scalar case \eqref{LEDeq},
nonexistence of bounded solutions of \eqref{LEDsys} was shown for any $p,q>1$,
thus again providing a complete answer to the question in the class of bounded solutions.

\goodbreak

\subsection{Main result}

Up to now, in view of the results recalled in the last paragraph,
the available nonexistence results for possibly unbounded solutions of system
\eqref{LEDsys} require strong limitations on $p,q$
(for instance $p,q<p_S$, or \eqref{hypLE3} with additional restrictions when $n\ge 5$).
Our goal is to prove nonexistence for any $p, q>1$
without requiring global boundedness of the solution.
Here is our main result.

\begin{theorem}\label{Theorem 2}
Let $p, q>1$. Then problem \eqref{LEDsys} does not admit any positive classical solution that is bounded on 
finite strips.
\end{theorem}

Here a finite strip is the set
$$\Gamma_R:=\{x\in \mathbb{R}^{n}:0<x_n<R\},\quad R>0,$$
and a classical solution is a solution 
with $u, v\in C(\overline{\mathbb{R}_+^{n}})\cap C^{2}(\mathbb{R}_+^{n})$.
We stress that no growth restriction at infinity is made
and our assumption just means that $(u,v)$ does not blow up at finite distance from the boundary.
Throughout this paper we always assume $p, q>1$.

\subsection{Plan of the proof}

For clarity, we will split the proof of Theorem \ref{Theorem 2} into the following four propositions.

\begin{proposition} \label{prop2}
Let $(u,v)$ be a positive classical solution of system  \eqref{LEDsys} which is bounded in finite strips.
Then $u_{x_n}$, $v_{x_n}>0$ in $\mathbb{R}_{+}^{n}$.
\end{proposition}

\begin{proposition}\label{prop3}
Let $(u,v)$ be a positive classical solution of system  \eqref{LEDsys} which is bounded in finite strips.
Then
\be{hypstrips}
\hbox{$\frac{\nabla u_{x_n}}{u_{x_n}}$ and $\frac{\nabla v_{x_n}}{v_{x_n}}$ are bounded on finite strips.}
\ee
\end{proposition}

\begin{proposition}\label{prop4}
Let $(u,v)$ be a positive classical solution of system  \eqref{LEDsys} which is bounded in finite strips.
Suppose that \eqref{hypstrips} holds. 
Then 
\be{uvconvex}
u_{x_nx_n}, v_{x_nx_n}\ge 0\quad\hbox{in $\mathbb{R}_{+}^{n}$.}
\ee
\end{proposition}

\begin{proposition}\label{prop5}
Let $(u,v)$ be a nonnegative classical solution of system  \eqref{LEDsys} such that 
 $u_{x_n}, v_{x_n},$ $u_{x_nx_n}, v_{x_nx_n}\ge 0$ in $\mathbb{R}_{+}^{n}$.
Then $u=v\equiv 0$.
\end{proposition}

Let us now explain the main ideas of the proofs and the novelties with respect to previous work.
The strategy 
from \cite{CLZ}, developed there for bounded solutions of both scalar 
equations and systems (more general than the Lane-Emden system), consists in showing that positive solutions,
which are increasing in $x_n$, as a consequence of the moving planes method,
must also be convex in the direction $x_n$, a property that leads relatively easily to a contradiction.
 To this end, one of the main ingredients of \cite{CLZ} is the auxiliary function
 \be{defeta1}
 \eta_1=\frac{u_{x_n x_n}}{(1+x_n)u_{x_n}}
 \ee
 and its analogue $\eta_2$ for $v$ (in case of systems),
which turn out to verify an elliptic system that is amenable to some kind of sophisticated maximum principle arguments. 
The boundedness of the solution is used in a crucial way in the proof of \cite{CLZ}
(see in particular the proof of claim (3.7)),
so as to show that $\eta_i$ decays as $x_n\to\infty$ and thereby ensure
that an eventual negative infimum of $\eta_i$ could only occur at
finite distance from the boundary.
Once this is done, a contradiction is reached by means of delicate arguments involving Harnack inequalities 
for coupled systems and several limiting procedures, which are needed to cope with the lack of compactness in the tangential direction that prevents direct application of the maximum principle. 
\smallskip

Although we follow the general convexity strategy from \cite{CLZ}, 
nontrivial new ideas are required in order to handle solutions with arbitrarily fast growth as $x_n\to\infty$.
In the scalar case, treated in \cite{DSS}, the key new idea was to notice that $\eta_1$ satisfies 
a nonlinear elliptic inequality with a weighted diffusion operator (with weight involving $(u_{x_n})^2$), and to apply a novel, nonlinear version of the maximum principle,
obtained by a suitable Moser type iteration argument, combined with variational stability estimates to control the $L^1$ growth of the weight. 
Unfortunately this approach does not seem to apply
to the system satisfied by $(\eta_1,\eta_2)$ when $(u,v)$ is a solution of the Lane-Emden system.

\smallskip

Instead we proceed in two main steps, corresponding to Propositions \ref{prop3} and \ref{prop4}, 
whose proofs are rather involved (Proposition \ref{prop2} and \ref{prop5} are much easier).

\smallskip

\begin{itemize} 

  \item[$\bullet$]  To prove Proposition~\ref{prop4}, we construct a decaying function $\varphi$, defined in terms of the quantity $h$ from \eqref{defhtau2} below, 
  such that, instead of $\eta_1,\eta_2$
  (cf.~\eqref{defeta1}), the modified functions
$$\tilde\eta_1=\varphi(x_n)\frac{u_{x_nx_n}}{u_{x_n}},\quad \tilde\eta_2=\varphi(x_n)\frac{ v_{x_nx_n}}{v_{x_n}}$$
decay as $x_n\to\infty$ and satisfy a ``good'' auxiliary elliptic system (see~Section~2).
Namely, assuming the logarithmic gradient bound on finite strips, i.e.:
 \be{defhtau2}
h(R):=\sup_{\Gamma_R}\left\{\frac{ |\nabla u_{x_n}|}{u_{x_n}}+\frac{ |\nabla v_{x_n}|}{v_{x_n}}\right\}<\infty,
\quad\hbox{for all $R>0$,}
\ee
 an appropriate, and rather delicate choice turns out to be given by
$$\varphi(s)=\int_{s}^{\infty}\int_{z}^{\infty}\frac{e^{-\tau}}{1+\hat h(\tau)}d\tau dz+c\bigl(1+\delta(s_0-s)\bigr)_+^3
\quad\hbox{where } \hat h(\tau)=\int_\tau^{\tau+1} h(s)\,ds,$$
for suitable $c,\delta,s_0>0$.
The maximum principle then allows to conclude that $\tilde\eta_1, \tilde\eta_2\ge 0$, hence $u_{x_nx_n},v_{x_nx_n}\ge 0$
(actually, the maximum principle is applied to a suitable tangential perturbation of $\tilde\eta_1, \tilde\eta_2$,
whose effect is to localize the minimum at a finite point, and the cost of the perturbation
can be absorbed by using \eqref{defhtau2} once more).

\smallskip

  \item[$\bullet$] The proof of Proposition \ref{prop3} (i.e.,~property \eqref{defhtau2}) is involved and relies on several ingredients:
\begin{itemize} 

\vskip 1pt

 \item[-]  the comparison property on components: ${v^{p+1}\over p+1} \le {u^{q+1}\over q+1}$
 (see~Subsection~\ref{secComp}),
obtained by combining ideas from \cite{BVY, So09, MSS};

\vskip 1pt

\item[-] a bound from below for $\frac{u_{x_nx_n}}{u_{x_n}}$ and $\frac{v_{x_nx_n}}{v_{x_n}}$ 
on finite strips, obtained by means of the Moser type iteration argument from \cite{DSS} 
(see~Subsections~\ref{secNLMP} and \ref{SecBelow});
\vskip 1pt

  \item[-] various Harnack type arguments (boundary Harnack inequality for $u,v$, Harnack inequalities for $u_{x_n}+v_{x_n}$
  and then $u_{x_n},v_{x_n}$, comparison of local infima of $u_{x_n}$ and $v_{x_n}$ by means of Green kernel representation;~see~Subsections~\ref{secHar1} and \ref{secHar2}).
\end{itemize}
 \end{itemize}

\smallskip
The organization of the paper is as follows. In section~2 we introduce some auxiliary functions
which are instrumental in the proof of both Propositions \ref{prop3} and \ref{prop4}.
In section~3 we prove Proposition \ref{prop4}, assuming Proposition \ref{prop3} is established.
Section~4 contains some preliminaries to the proof of Proposition~\ref{prop3}, namely comparison of components and 
a nonlinear maximum principle. Proposition \ref{prop3} is then proved in Section~5.
Finally, Section~6 contains the short proofs of Propositions \ref{prop2} and \ref{prop5} and of Theorem \ref{Theorem 2}.

\section{Auxiliary functions}

In this section, we introduce some auxiliary functions,
which will be instrumental in the proof of both Propositions \ref{prop3} and \ref{prop4}.
The following lemma is stated for general function $\varphi$, and appropriate choices will be made in subsequent sections.

\begin{lemma}\label{ineqL1L2}
Let $(u,v)$ be a positive classical solution of system  \eqref{LEDsys}.
Let $\varphi=\varphi(x_n)\in C^{2}([0,\infty))$ be a positive function. Set 
$$\eta_1=\varphi(x_n)\frac{ u_{x_nx_n}}{u_{x_n}},\quad 
\eta_2=\varphi(x_n)\frac{ v_{x_nx_n}}{v_{x_n}}.$$
Then
\be{reguleta}
\eta_1, \eta_2\in C^2(\Rnp)\cap C^1(\overline\Rnp)
\ee
and $\eta_1, \eta_2$ satisfy
$$
L_1\eta_1:=-\Delta  \eta_1-b_1\cdot\nabla \eta_1\ge a\eta_1^{2}+d\eta_1
+\frac{pv^{p-1}v_{x_n}}{u_{x_n}}(\eta_2-\eta_1)
\quad \hbox{in $\Rnp$}
$$
and
$$
L_{2}\eta_2:=-\Delta \eta_2-b_2\cdot\nabla\eta_2\ge a\eta_2^{2}+d\eta_2
+\frac{qu^{q-1}u_{x_n}}{v_{x_n}}(\eta_1-\eta_2)
\quad \hbox{in $\Rnp$,}
$$
where
$$b_1=2\left(\frac{ \nabla u_{x_n}}{u_{x_n}}-e_{n}\frac{\varphi'}{\varphi}\right),\quad b_2=2\left(\frac{ \nabla v_{x_n}}{v_{x_n}}-e_{n}\frac{\varphi'}{\varphi}\right),
\quad a=-2\frac{\varphi'}{\varphi^{2}}, \quad d=\frac{2\varphi'^{2}-\varphi\varphi''}{\varphi^{2}}.$$
\end{lemma}

\smallskip

\begin{proof}
Since $p,q\ge 1$, by elliptic regularity we have 
$u,v\in C^4(\Rnp)\cap C^3(\overline\Rnp)$, hence \eqref{reguleta}.
By differentiating twice with respect to $x_n$, we obtain, in $\Rnp$:
\be{deriv1}
-\Delta u_{x_n}=pv^{p-1}v_{x_n},
\qquad 
-\Delta v_{x_n}=qu^{q-1}u_{x_n}
\ee
and
\be{deriv2}
-\Delta u_{x_nx_n}=p(p-1)v^{p-2}v^{2}_{x_n}+pv^{p-1}v_{x_nx_n},
\qquad
-\Delta v_{x_nx_n}=q(q-1)u^{q-2}u^{2}_{x_n}+qu^{q-1} u_{x_nx_n}.
\ee
Set
$$h_{1}=\frac{ u_{x_nx_n}}{u_{x_n}},\quad h_{2}=\frac{ v_{x_nx_n}}{v_{x_n}},$$
it follows that
$$
\nabla h_{1}=\frac{ u_{x_n}\nabla  u_{x_nx_n}-u_{x_nx_n}\nabla u_{x_n}}{u^{2}_{x_n}},
\qquad
\nabla h_{2}=\frac{ v_{x_n}\nabla  v_{x_nx_n}-v_{x_nx_n}\nabla v_{x_n}}{v^{2}_{x_n}}.
$$
Let us now compute
\[
\begin{split}
\Delta h_{1}&=\frac{ u_{x_n}\Delta  u_{x_nx_n}+\nabla u_{x_n}\cdot\nabla  u_{x_nx_n}-u_{x_nx_n}\Delta u_{x_n}-\nabla u_{x_nx_n}\cdot\nabla u_{x_n}}{u^{2}_{x_n}}\\
&\qquad -2\frac{ (u_{x_n}\nabla  u_{x_nx_n}-u_{x_nx_n}\nabla u_{x_n})\cdot \nabla u_{x_n}}{u^{3}_{x_n}}
=\frac{ u_{x_n}\Delta  u_{x_nx_n}-u_{x_nx_n}\Delta u_{x_n}}{u^{2}_{x_n}}-2\frac{ \nabla u_{x_n}}{u_{x_n}}\cdot \nabla h_{1}.
\end{split}
\]
Using \eqref{deriv1}-\eqref{deriv2}, we have
$$
u_{x_n}\Delta  u_{x_nx_n}-u_{x_nx_n}\Delta u_{x_n}
=-p(p-1)v^{p-2}v^{2}_{x_n}u_{x_n}-pv^{p-1}v_{x_nx_n}u_{x_n}+pv^{p-1}v_{x_n}u_{x_nx_n},
$$
so that
$$
\Delta h_{1}=-\frac{p(p-1)v^{p-2}v^{2}_{x_n}}{u_{x_n}}+\frac{pv^{p-1}v_{x_n}}{u_{x_n}}\left(\frac{ u_{x_nx_n}}{u_{x_n}}-\frac{ v_{x_nx_n}}{v_{x_n}}\right)-2\frac{ \nabla u_{x_n}}{u_{x_n}}\cdot \nabla h_{1}.
$$
Since $\eta_1=\varphi\frac{ u_{x_nx_n}}{u_{x_n}}=\varphi h_{1}$, we get
$\nabla \eta_1=\varphi \nabla h_{1}+\varphi' h_{1}e_{n}$
and
\[
\begin{split}
-\Delta \eta_1&=-\varphi \Delta h_{1}-2\varphi' \nabla h_{1}\cdot e_{n}-\varphi'' h_{1}\\
&=\varphi\left(\frac{p(p-1)v^{p-2}v^{2}_{x_n}}{u_{x_n}}-\frac{pv^{p-1}v_{x_n}}{u_{x_n}}(h_{1}-h_{2})\right)+2\varphi\frac{ \nabla u_{x_n}}{u_{x_n}}\cdot \nabla h_{1}-2\varphi' \nabla h_{1}\cdot e_{n}-\frac{\varphi''}{\varphi}\eta_{1}.
\end{split}
\]
Noticing that 
$$
2\varphi \frac{ \nabla u_{x_n}}{u_{x_n}}\cdot \nabla h_{1}-2\varphi' \nabla h_{1}\cdot e_{n}
=2\left(\frac{ \nabla u_{x_n}}{u_{x_n}}-\frac{\varphi'}{\varphi}e_{n}\right) \cdot\varphi\nabla h_{1}
=2\left(\frac{ \nabla u_{x_n}}{u_{x_n}}-\frac{\varphi'}{\varphi}e_{n}\right) \cdot(\nabla \eta_1-\varphi' h_{1}e_{n})
$$
and
$$
2\left(\frac{ \nabla u_{x_n}}{u_{x_n}}-\frac{\varphi'}{\varphi}e_{n}\right) \cdot(-\varphi' h_{1}e_{n})
=-2\left(\frac{  u_{x_nx_n}}{u_{x_n}}-\frac{\varphi'}{\varphi}\right) \frac{\varphi' }{\varphi}\eta_1
=-2\frac{\varphi' }{\varphi^{2}}\eta_1^{2}+2 \frac{\varphi'^{2}}{\varphi^{2}}\eta_1,
$$
we obtain
$$-\Delta \eta_1-b_1\cdot\nabla \eta_1=
{\frac{p(p-1)\varphi v^{p-2}v^{2}_{x_n}}{u_{x_n}}+\frac{pv^{p-1}v_{x_n}}{u_{x_n}}(\eta_2-\eta_1)}-2\frac{\varphi' }{\varphi^{2}}\eta_1^{2}+\frac{2\varphi'^{2}-\varphi''\varphi}{\varphi^{2}}\eta_1.
$$
Exchanging the roles of $u, v$, we get the corresponding expression for $v$ and,
since $u_{x_n}, v_{x_n}>0$ and $p, q\ge 1$, the conclusion follows.
\end{proof}

\section{Convexity in the normal direction: proof of Proposition \ref{prop4}}

In this section, we prove Proposition \ref{prop4} assuming Proposition \ref{prop3} is established. We choose this order because the proof of Proposition \ref{prop4} is shorter,
and because the connection of this step with the final result is easier to figure out.
Proposition \ref{prop3} will then be proved (independently) in the next two sections.

\medskip

\begin{proof}[Proof of Proposition \ref{prop4}]
Assume for contradiction that \eqref{uvconvex} fails.
Then there exists $\hat{x}\in \mathbb{R}_{+}^{n}$ such that 
\be{defsigma}
\sigma:=-\min\left\{\frac{u_{x_nx_n}}{u_{x_n}}(\hat{x}),\frac{v_{x_nx_n}}{v_{x_n}}(\hat{x})\right\}>0.
\ee
We shall modify $\frac{u_{x_nx_n}}{u_{x_n}}, \frac{v_{x_nx_n}}{v_{x_n}}$ in an appropriate way,
applying Lemma~\ref{ineqL1L2} with a suitable choice of $\varphi$ and using a perturbation argument,
so as to produce functions which satisfy a ``good'' elliptic system and are well behaved at infinity,
to which one can apply a maximum principle argument.
To this end, owing to Proposition \ref{prop3}, we may first define 
\be{defhtau}
h(R):=\sup_{\Gamma_R}\left\{\frac{ |\nabla u_{x_n}|}{u_{x_n}}+\frac{ |\nabla v_{x_n}|}{v_{x_n}}\right\}<\infty,
\quad R>0,
\ee
which is nondecreasing in $R$. 
Let $\hat h(R)$ be a nondecreasing {\it continuous} function such that $\hat h(R)\ge h(R)$
(one can take for instance $\hat h(R)=\int_R^{R+1} h(\tau)\,d\tau$).
We then set
$$\eta_1=\varphi(x_n)\frac{ u_{x_nx_n}}{u_{x_n}},\quad \eta_2=\varphi(x_n)\frac{v_{x_nx_n}}{v_{x_n}},$$
with $\varphi\in C^2([0,\infty))$ given by
\be{defvarphi}
\varphi(s)=\varphi_1(s)+\varphi_2(s),
\quad 
\varphi_1(s)=\int_{s}^{\infty}\int_{z}^{\infty}\frac{e^{-\tau}}{1+\hat h(\tau)}d\tau dz,
\quad
\varphi_2(s)=3\sigma^{-1}\bigl(1+\delta(s_0-s)\bigr)_+^3,
\ee
where $s_0=\hat x_n$ and $\delta>0$ is chosen small enough so that
\be{defvarphi2}
\sup_{[0,\infty)} |\varphi_2'|=9\sigma^{-1}\delta(1+\delta s_0)^2\le 1.
\ee

We claim that
\be{convexvarphi}
-2\le \varphi'<0,\quad \varphi''\ge 0\quad\hbox{on $[0,\infty)$}
 \ee
and
\be{claimxnlarge}
\lim_{x_n\to \infty} \biggl(\,\sup_{x'\in\R^{n-1}} \bigl(|\eta_1(x',x_n)|+|\eta_2(x',x_n)|\bigr)\biggr)=0.
\ee
Using that $\hat h$ is a nondecreasing function, we have, for all $s\ge 0$,
\be{varphiprime}
0<-\varphi_1'(s)=\int_{s}^{\infty}\frac{e^{-\tau}}{1+\hat h(\tau)}d\tau\le \frac{1}{1+\hat h(s)}\int_{s}^{\infty}e^{-\tau}d\tau
=\frac{e^{-s}}{1+\hat h(s)}
\ee
and
$$\begin{aligned}
h(s)\varphi_1(s)
&=-h(s)\int_{s}^{\infty}\varphi'_{1}(\tau)d\tau\le \hat h(s)\int_{s}^{\infty}\frac{e^{-\tau}}{1+\hat h(\tau)}d\tau 
\le \frac{\hat h(s)}{1+\hat h(s)}\int_{s}^{\infty}e^{-\tau}d\tau\le e^{-s}.
\end{aligned}$$
Since also $\varphi_2(s)=0$ for $s\ge s_0+\delta^{-1}$,
we deduce that $\lim_{s\to \infty}h(s)\varphi(s)=0$, hence in particular \eqref{claimxnlarge},
whereas \eqref{convexvarphi} easily follows from \eqref{defvarphi}, \eqref{defvarphi2}, \eqref{varphiprime}.

\smallskip

Now we note that, since $\varphi(s_0)\ge \varphi_2(s_0)=3\sigma^{-1}$ and $s_0=\hat x_n$, we deduce from \eqref{defsigma} that
\be{defsigma2}
\min\bigl\{\eta_1(\hat{x}),\eta_2(\hat{x})\bigr\}\le -3.
\ee
In view of \eqref{claimxnlarge}, there exists $\bar R>\hat x_n$ such that
\be{defRlarge}
\eta_i(x',x_n)\ge -2\quad\hbox{for all $(x',x_n)\in\R^{n-1}\times[\bar R,\infty)$ and $i\in\{1,2\}$.}
\ee
For any $\eps>0$, 
we next define the tangentially perturbed functions
$$\phi^i_{\eps}(x)=\eta_i(x)+\eps \theta(x),\ i\in\{1,2\}, \quad\hbox{with $\theta=\ln (1+|x-\hat{x}|^{2})\ge 0$.}$$
By \eqref{defhtau} for $s=\bar R$, \eqref{defsigma2} and \eqref{defRlarge}, we deduce that
\be{mininf}
\min\Bigl\{\inf_{\Rn} \phi^1_\eps, \inf_{\Rn}\phi^2_\eps\Bigr\}
=\min\Bigl\{\inf_{\Gamma_{\bar R}} \phi^1_\eps, \inf_{\Gamma_{\bar R}}\phi^2_\eps\Bigr\}
\in(-\infty,-3].
\ee
Since $u=v=0$ on $\partial \mathbb{R}_+^{n}$, we have $\frac{\partial^{2}u}{\partial x^2_j}=\frac{\partial^{2}v}{\partial x^2_j}=0$ on $\partial \mathbb{R}_+^{n}$ for $j=1,\cdots, n-1$, hence $\frac{\partial^{2}u}{\partial x^2_n}=\frac{\partial^{2}v}{\partial x^2_n}=0$, so that 
\be{BCphi}
\phi^{i}_{\eps}\ge 0 \quad\hbox{on $\partial \mathbb{R}_{+}^{n}$,\quad $i\in\{1,2\}$.}
\ee
Using again \eqref{defhtau} for $s=\bar R$, we have 
\be{minxeps}
\lim_{|x'|\to\infty}\left( \inf_{x_n\in[0,\bar R]}\phi^i_\eps(x',x_n)\right)=\infty,\quad i\in\{1,2\}.
\ee
From \eqref{defRlarge}-\eqref{minxeps}, 
it follows that, for each $\eps>0$, there exists $\hat x^\eps\in \Gamma_{\bar R}$ such that 
$$\min\bigl\{\phi^1_\eps(\hat x^\eps),\phi^2_\eps(\hat x^\eps)\bigr\}=
\min\Bigl\{\inf_{\Gamma_{\bar R}} \phi^1_\eps, \inf_{\Gamma_{\bar R}}\phi^2_\eps\Bigr\}$$
and we may assume without loss of generality that 
\be{compphi}
\phi^1_\eps(\hat x^\eps)\le\phi^2_\eps(\hat x^\eps)
\ee
(indeed, the other case can be treated similarly,
by using the equation for $\phi^2_\eps$ instead of $\phi^1_\eps$ in the next paragraph).

\smallskip

We then consider the PDE satisfied by $\phi^1_\eps$. By Lemma~\ref{ineqL1L2}, we have
$$L_1\eta_1=-\Delta \eta_1-b_1\cdot\nabla\eta_1\ge a\eta_1\left(\eta_1+\frac{d}{a}\right)
+\frac{pv^{p-1}v_{x_n}}{u_{x_n}}(\eta_2-\eta_1)$$
with
\be{L1phi0}
b_1=2\left(\frac{ \nabla u_{x_n}}{u_{x_n}}-e_{n}\frac{\varphi'}{\varphi}\right),\quad a=-2\frac{\varphi'}{\varphi^{2}}, \quad d=\frac{2\varphi'^{2}-\varphi\varphi''}{\varphi^{2}},
\ee
hence
\be{L1phi}
L_1\phi^1_\eps=L_1\eta_1+\eps L_1\theta\ge 
a\eta_1\left(\eta_1+\frac{d}{a}\right)+\frac{pv^{p-1}v_{x_n}}{u_{x_n}}(\phi^2_\eps-\phi^1_\eps)+\eps L_1\theta.
\ee
In view of \eqref{convexvarphi}, we obtain
\be{da1}
\frac{d}{a}=-\frac{2\varphi'^{2}-\varphi\varphi''}{2\varphi'}=-\varphi'+\frac{\varphi\varphi''}{2\varphi'}\le -\varphi'\le 2.
\ee
Thus \eqref{mininf} and \eqref{da1} guarantee that
\be{da3}
\eta_1\le\phi^1_\eps\le -3 \quad\hbox{and}\quad
\eta_1+\frac{d}{a}\le -1 \quad\hbox{at $x=\hat x^\eps$.}
\ee
On the other hand, elementary computations show that $\sup_{\R^n} (|\nabla\theta|+|\Delta\theta|)<\infty$.
Since $\sup_{\Gamma_{\bar R}}|b_1|<\infty$ by  \eqref{L1phi0} and the log-grad estimate \eqref{defhtau},
we can control the cost of the tangential perturbation $\theta$, namely:
$$K:=\sup_{\Gamma_{\bar R}}|L_1 \theta|<\infty.$$
Setting $\kappa:=\inf_{s\in[0,\bar R]} a(s)>0$ (cf.~\eqref{convexvarphi}, \eqref{L1phi0} and combining \eqref{compphi}, \eqref{L1phi} and \eqref{da3}, it follows that
$$
0\ge L_1\phi^{1}_{\eps}(\hat x^\eps)
\ge \left\{a\eta_1\left(\eta_1+\frac{d}{a}\right)+\eps L_1\theta\right\}(\hat x^\eps)
\ge 3\kappa-K\eps.
$$
Since $\kappa, K, \sigma$ are independent of $\eps$, this is a contradiction for $\eps>0$ sufficiently small.
Proposition \ref{prop4} is proved.
\end{proof}

\section{Preliminaries to the proof of Proposition \ref{prop3}}

\subsection{Comparison of components} \label{secComp}

Comparison of components for the Lane-Emden system was first used in \cite{BVY} in bounded domains
and then in \cite{So09} in the case of $\Rn$.
In the half-space case, this property was studied in \cite{MSS} for other systems 
(with nonlinearities of equal homogeneity in each equation).
We here extend the property to the Lane-Emden system in a half-space by combining ideas from \cite{BVY, So09, MSS}.

\begin{proposition}\label{comparison property}
Suppose $p\ge q$ and let $(u,v)$ be
a positive classical solution of system \eqref{LEDsys}.
Then 
\be{compcomp}
{v^{p+1}\over p+1} \le {u^{q+1}\over q+1}\quad\quad \mbox{\rm in }\  \mathbb{R}_+^{n}.
\ee
\end{proposition}

The proof relies on half-spherical means.
Recall that the half-spherical means of a function $w\in C(\overline{\mathbb{R}_+^{n}})$
are defined by:
$$[w](R)=\frac{1}{R^{2}|\mathbb{S}^{+}_{R}|}\int_{\mathbb{S}^{+}_{R}}w(x)x_nd\sigma_{R}(x),
\quad R>0,$$
where $\mathbb{S}^{+}_{R}=\{x\in \mathbb{R}_+^{n},\ |x|=R\}$.
We shall use the following properties 
of half-spherical means, respectively a lower bound for nonnegative superharmonic functions 
and a maximum principle of Phragm\'en-Lindel\"of type (see \cite[Lemmas 5.2 and 5.3]{MSS}).

\begin{lemma}\label{halfspherical}
(i) Let $v\in C^{2}(\mathbb{R}_+^{n})\cap C(\overline{\mathbb{R}_+^{n}})$ be nonnegative and superharmonic in $\mathbb{R}_+^{n}$. Then the function $R\to [v](R)$ is nonincreasing and 
there exists a constant $c_0=c_0(n)>0$ such that the limit $L(v)=\lim_{R\to \infty}[v](R)\in [0,\infty)$
satisfies:
$$v(x)\ge c_0 L(v)x_n\quad \mbox{\rm in }\  \mathbb{R}_+^{n}.$$
(ii) Let $w\in C^{2}(\mathbb{R}_+^{n})\cap C(\overline{\mathbb{R}_+^{n}})$ satisfy $\Delta w\ge 0$ on the set $\{w\ge 0\}$
and $w\le 0$ on $\partial \mathbb{R}_+^{n}$. If 
$$\liminf_{R\to \infty}[w_{+}](R)=0,$$
where  $w_{+}=\max\{w,0\}$, then $w\le 0$ in $\mathbb{R}_+^{n}$. 
\end{lemma}

We shall also use the following Liouville type result for weighted elliptic inequalities,
which is a special case of \cite[Lemma 3.1]{MSS} (see also \cite{AS}).

\begin{lemma}\label{Liouville result for weighted elliptic inequalities}
Let $r\ge 0$ and let $u\in C^2(\mathbb{R}_+^{n})$ be a nonnegative solution of
$$-\Delta u\ge |x|^\kappa\, \chi_\Sigma\, u^r \quad \mbox{\rm in }\  \mathbb{R}_+^{n},$$
where $\Sigma=\bigl\{x:x_n\ge \delta|x|\bigr\}$, 
$\kappa>-2$, $\kappa+r\ge -1$ and $c,\delta>0$. If 
$$0\le r\le\frac{n+1+\kappa}{n-1},$$
then $u\equiv 0$.
\end{lemma}

\begin{proof}[Proof of Proposition~\ref{comparison property}]
Let $\sigma=\frac{q+1}{p+1}\le 1$, $\ell =\sigma^{-\frac{1}{p+1}}$ and $w=v-\ell u^{\sigma}$. A direct calculation gives that
$$\Delta w
=\Delta  v-\ell \Delta u^{\sigma} =\Delta  v-\ell \sigma( u^{\sigma-1}\Delta u+(\sigma-1)u^{\sigma-2}|\nabla u|^{2} 
\ge -u^{q}+\ell \sigma u^{\sigma-1} v^{p} =u^{\sigma-1}\left(\Bigl(\frac{v}{\ell}\Bigl)^{p}-u^{p\sigma }\right).$$
It follows that
\be{liminf0}
\Delta w\ge 0 \quad \mbox{\rm on the set }\{w\ge 0\}.
\ee

Next we claim that
\be{liminf}
\liminf_{R\to \infty}[w_{+}](R)=0.
\ee
Assume for contradiction that \eqref{liminf} fails. Then
$$\liminf_{R\to \infty}[v](R)\ge\liminf_{R\to \infty}[w_{+}](R)>0.$$
Since $v$ is nonnegative and superharmonic in $\mathbb{R}_+^{n}$, 
it follows from Lemma~\ref{halfspherical}(i) that 
$v(x)\ge cx_n$ in $\mathbb{R}_+^{n}$ with $c=c_0(n)\lim_{R\to \infty}[v](R)>0$, hence
$$-\Delta u=v^{p}\ge c^{p}x^{p}_{n}\quad \mbox{\rm in }\  \mathbb{R}_+^{n}.$$
But Lemma~\ref{Liouville result for weighted elliptic inequalities} with $r=0$ and $\kappa=p$
implies $u\equiv 0$: a contradiction. So we obtain \eqref{liminf}.

Finally, it follows from \eqref{liminf0}, \eqref{liminf} and Lemma~\ref{halfspherical}(ii) that $w\le0$, that is \eqref{compcomp}.
\end{proof}

\subsection{A nonlinear maximum principle in strips} \label{secNLMP}

We shall use the following lemma, which is a variant in finite strips 
of a result from \cite{DSS} in the half-space;
see [Lemma~3.1 and formula (4.7) in \cite{DSS} 
(we adopt a different sign convention for convenience).

\begin{lemma}\label{the main lemma of DSS}
Let $k>1$, $\rho>0$, $A\in L^{\infty}(\Gamma_\rho)$, with $A>0$ a.e. in $\Gamma_\rho$,
and consider the elliptic operator given by
$$\mathcal{L}=A^{-1}\nabla\cdot(A\nabla).$$
If $\xi \in H_{loc}^{1}\cap C(\overline{\Gamma_\rho})$ is a weak solution of
\be{Lxiweak}
\mathcal{L}\xi\ge (\xi_{+})^{k}\quad\mbox{in}\quad  \Gamma_\rho,
\qquad\hbox{$\xi\le 0$ on $\partial \Gamma_\rho$,}
\ee
then $\xi\le 0$.
\end{lemma}

\begin{rem} (i) Here $\xi$ being a weak solution of \eqref{Lxiweak} is understood in the following sense:
$$\int_{\Gamma_\rho}A (\xi_{+})^{k}\varphi\le  -\int_{\Gamma_\rho}A \nabla \xi\cdot  \nabla\varphi$$
for all $\varphi\in H^1(\overline\Gamma_\rho)$ such that $\varphi\ge 0$ and 
$Supp(\varphi)$ is a compact subset of $\overline{\Gamma_\rho}$.

\smallskip

(ii) Although the assumption $A\in L^{\infty}(\Gamma_\rho)$ will be sufficient for our needs, 
we could replace it by the weaker assumption 
$A\in L^{\infty}_{loc}(\overline{\Gamma_\rho})$ and $\log\bigl(\int_{0}^\rho\int_{R\le |x'|\le 2R} A\bigr)=o(R^2)$ as $R\to\infty$.
\end{rem} 

\smallskip

\begin{proof}
We claim that there exists a constant $C=C(n, k)>0$ such that for all $R>1$ and $m\ge \frac{k+1}{k-1}$, we have
\be{Lxiweak2}
\left(\int_{0}^\rho\int_{|x'|\le R} A(\xi_+)^{(k-1)m}\right)^{\frac{1}{m}}\le C\frac{m}{R^{2}}\left(\int_{0}^\rho\int_{R\le |x'|\le 2R} A\right)^{\frac{1}{m}}.
\ee
Here $x=(x',x_n)\in \mathbb{R}^{n-1}\times [0,\rho]$.
Set $\theta=k-1$, and denote $\int_{\Gamma_\rho}=\int$ for simplicity. Fix $\alpha\ge 1$ and let $\varphi\in C^{\infty}(\overline{\Gamma_\rho})$
have compact support. Since $\xi\le 0$ on $\partial \Gamma_\rho$, we have  $\xi_+=0$ on $\partial \Gamma_\rho$ and we may test  the equation $\mathcal{L}\xi\ge (\xi_+)^{k}$
with $\phi=(\xi_+)^{2\alpha-1}\varphi^{2}$. Using $\nabla \xi_+=\chi_{\{\xi>0\}}\nabla \xi$ a.e., this yields
\[
\begin{split}
\int A(\xi_+)^{2\alpha+\theta}\varphi^{2}&\le -\int A\nabla\xi\cdot\nabla [(\xi_+)^{2\alpha-1}\varphi^{2}]
=-\int A(\xi_+)^{2\alpha-1}\nabla\xi\cdot\nabla (\varphi^{2})-\int A\nabla\xi\cdot\nabla [(\xi_+)^{2\alpha-1}]\varphi^{2}\\
&=-\int A(\xi_+)^{2\alpha-1}\nabla\xi\cdot\nabla (\varphi^{2})-(2\alpha-1)\int A(\xi_+)^{2\alpha-2}|\nabla \xi_+|^{2}\varphi^{2}.
\end{split}
\]
Using Young's inequality, 
we have
\[
\begin{split}
-\int A(\xi_+)^{2\alpha-1}\nabla\xi\cdot\nabla (\varphi^{2})&=
-\int 2A(\xi_+)^{2\alpha-1}\varphi\nabla\xi_+\cdot\nabla \varphi\\ 
&=-2\int \sqrt{2\alpha-1}A^{\frac{1}{2}}(\xi_+)^{\alpha-1}\varphi\nabla\xi_+\cdot\frac{A^{\frac{1}{2}}(\xi_+)^{\alpha}\nabla \varphi}{\sqrt{2\alpha-1}}\\
&\le (2\alpha-1)\int A(\xi_+)^{2\alpha-2}\varphi^{2}|\nabla \xi_+|^{2}+\frac{1}{{2\alpha-1}}\int A(\xi_+)^{2\alpha}|\nabla\varphi|^{2}.
\end{split}
\]
It follows that 
\[
\int A(\xi_+)^{2\alpha+\theta}\varphi^{2}\le \frac{1}{{2\alpha-1}}\int A(\xi_+)^{2\alpha}|\nabla\varphi|^{2}.
\]
Let us now choose $\varphi$ of the form $\varphi=\psi^{m}$, with $m\ge \frac{k+1}{k-1}$ and $0\le \psi\in C^{\infty}(\overline{\Gamma_\rho})$ with
compact support. By H\"older's inequality, we have
$$\begin{aligned}
\frac{1}{{2\alpha-1}}\int A(\xi_+)^{2\alpha}|\nabla\varphi|^{2}
&=\frac{m^{2}}{{2\alpha-1}}\int A\psi^{2m-2}(\xi_+)^{2\alpha}|\nabla\psi|^{2}
=\frac{m^{2}}{{2\alpha-1}}\int A^{\frac{2\alpha}{2\alpha+\theta}}\psi^{2m-2}(\xi_+)^{2\alpha}A^{\frac{\theta}{2\alpha+\theta}}|\nabla\psi|^{2}\\
&\le\frac{m^{2}}{{2\alpha-1}}\left(\int A\psi^{\frac{(2\alpha+\theta)(m-1)}{\alpha}}(\xi_+)^{2\alpha+\theta}\right)^{\frac{2\alpha}{2\alpha+\theta}}\left(\int A|\nabla\psi|^{\frac{2(2\alpha+\theta)}{\theta}}\right)^{\frac{\theta}{2\alpha+\theta}},
\end{aligned}$$
which implies that
\[
\int A(\xi_+)^{2\alpha+\theta}\psi^{2m}\le\frac{m^{2}}{{2\alpha-1}}\left(\int A\psi^{\frac{(2\alpha+\theta)(m-1)}{\alpha}}(\xi_+)^{2\alpha+\theta}\right)^{\frac{2\alpha}{2\alpha+\theta}}\left(\int A|\nabla\psi|^{\frac{2(2\alpha+\theta)}{\theta}}\right)^{\frac{\theta}{2\alpha+\theta}}.
\]
Let $\alpha=\frac{\theta(m-1)}{2}$, then $2\alpha+\theta=\theta m$ and $2m=\frac{(2\alpha+\theta)(m-1)}{\alpha}$. Thus, we get
\[
\left(\int A(\xi_+)^{\theta m}\psi^{2m}\right)^{\frac{1}{m}}\le\frac{m^{2}}{\theta(m-1)-1}\left(\int A|\nabla\psi|^{2m}\right)^{\frac{1}{m}}.
\]
We now consider rescaled test-functions $\psi$ with cylindrical symmetry.
Namely, we fix a radially symmetric function  $\psi_{1}\in C^{\infty}(\mathbb{R}^{n-1})$ such that  $\psi_{1}(y')\ge 0$,  $\psi_{1}(y')=1$ for $|y'|\le 1$ and  $\psi_{1}(y')=0$ for $|y'|\ge 2$. 
For given $R>1$, we then set in the last inequality 
$\psi(x)=\psi_{1}\bigl(\frac{x'}{R}\bigr)$ for $x\in \mathbb{R}^n_+$. 
Recalling $\theta=k-1$, we then obtain
\[
\left(\int_{0}^\rho\int_{|x'|\le R} A(\xi_+)^{(k-1)m}\right)^{\frac{1}{m}}\le\frac{\|\nabla\psi_{1}\|^{2}_{\infty}}{R^{2}}\frac{m^{2}}{(k-1)m-k}\left(\int_{0}^\rho\int_{R\le |x'|\le 2R} A\right)^{\frac{1}{m}}
\]
and inequality \eqref{Lxiweak2} follows.

\smallskip

Let now $D\ge\frac{k+1}{k-1}$. Since $A\in L^{\infty}(\Gamma_\rho)$,
choosing $m=R\ge D$, we get that
\[
\begin{split}
&\left(\int_{0}^\rho\int_{|x'|\le D} A(\xi_+)^{(k-1)m}\right)^{\frac{1}{m}}\le \left(\int_{0}^\rho\int_{|x'|\le m} A(\xi_+)^{(k-1)m}\right)^{\frac{1}{m}}
\le \frac{Cm}{m^{2}}\left(\int_{0}^\rho\int_{m\le |x'|\le 2m} A\right)^{\frac{1}{m}}
\le \frac{Cm^{\frac{n-1}{m}}}{m},
\end{split}
\]
which tends to $0$ as $m\to\infty$.
But since  $A> 0$ a.e.~in $\mathbb{R}_+^{n}$, 
we have
$$\lim_{m\to\infty}\left(\int_0^\rho\int_{|x'|\le D} A(\xi_+)^{(k-1)m}\right)^{\frac{1}{m}}
= \|(\xi_+)^{k-1}\|_{L^{\infty}(\Gamma_\rho\cap\{|x'|\le D\})}$$
and, since $D\ge\frac{k+1}{k-1}$ is arbitrary, we conclude that $\xi\le 0$ in $\Gamma_\rho$.
\end{proof}

\section{Log-grad estimate for $u_{x_n}$ and $v_{x_n}$: proof of Proposition \ref{prop3}}

It is somewhat involved and will require several steps.

\subsection{A Log-grad min estimate} 
\label{secHar1}

\begin{lemma}\label{claim1}
Let $(u,v)$ be a positive classical solution of system  \eqref{LEDsys} which is
bounded on finite strips. 
Then
$$\min\left\{\frac{|\nabla u_{x_n}|}{u_{x_n}},\frac{|\nabla v_{x_n}|}{v_{x_n}}\right\} \hbox{ is bounded on finite strips.}$$
\end{lemma}

\begin{proof}
We first extend the equation to the whole space by odd reflection. 
Namely, we set 
\be{oddrefl}
u(x',x_n)=-u(x',-x_n), \quad  v(x',x_n)=-v(x',-x_n),\quad x'\in \mathbb{R}^{n-1},\ x_n<0.
\ee
Using $u=v=0$ on $\partial\Rnp$, it is easy to see that
$u,v\in C^{2}(\Rn)$ (see, e.g.,~\cite[p.55]{QS19} for details) and that
$$
\begin{cases}
\begin{aligned}
&-\Delta u=|v|^{p-1}v&\quad&\mbox{\rm in }\ \Rn,\\
&-\Delta v=|u|^{q-1}u&\quad&\mbox{\rm in }\ \Rn.
\end{aligned}
\end{cases}
$$
Moreover, denoting $\Sigma_R:=\{x\in\Rn;\ |x_n|<R\}$, the assumption of boundedness of $u,v$ in finite strips guarantees that
\be{bddnessuv}
\sup_{\Sigma_R}(|u|+|v|)<\infty,\quad R>0.
\ee

Next, since $p,q\ge 1$, by elliptic regularity, we see that $u,v\in C^3(\Rn)$ and we obtain that
\be{eqnsum}
-\Delta u_{x_n}=p|v|^{p-1}v_{x_n}\quad\hbox{and}\quad-\Delta v_{x_n}=q|u|^{p-1}u_{x_n}\quad\mbox{\rm in }\ \Rn.
\ee
Consequently, 
$$ 
-\Delta (u+v)_{x_n}=\frac{p|v|^{p-1}v_{x_n}+q|u|^{q-1}u_{x_n}} {u_{x_n}+v_{x_n}}(u_{x_n}+v_{x_n})\quad\mbox{\rm in }\  \Rn.
$$
Let $R>1$, using $u_{x_n}, v_{x_n}\ge 0$ and \eqref{bddnessuv}, we have  
$$ 
\frac{p|v|^{p-1}v_{x_n}+q|u|^{q-1}u_{x_n}} {u_{x_n}+v_{x_n}} \le p|v|^{p-1}+q|u|^{q-1}\le C(R)<\infty \quad\mbox{\rm in }\ \Sigma_{R+2}. 
$$
Fix any $\xi\in\Sigma_R$. Applying the Harnack inequality to $(u+v)_{x_n}$ on $B_{1}(\xi)$, we obtain
\be{Harnack inequality}
\sup_{B_{1}(\xi)}u_{x_n}+\sup_{B_{1}(\xi)}v_{x_n}\le 2\sup_{B_{1}(\xi)}(u+v)_{x_n}\le C_1\inf_{B_{1}(\xi)}(u+v)_{x_n}
\ee
for some $C_1=C_1(R)>0$ ($C_1$ and $C_2$ below may also depend on the solution but are independent of $\xi$).
Next applying elliptic estimates to the first equation in \eqref{eqnsum} and then using \eqref{Harnack inequality}, it follows that 
$$
|\nabla u_{x_n}(\xi)|\le C_2\biggl(\sup_{B_{1}(\xi)}u_{x_n}+\sup_{B_{1}(\xi)}v_{x_n}\biggr)
\le C_1C_2\inf_{B_{1}(\xi)}(u+v)_{x_n}\le C_1C_2(u+v)_{x_n}(\xi),
$$
for some $C_2=C_2(R)>0$,
and we obtain similarly that
$$
|\nabla v_{x_n}(\xi)|\le C_1C_2(u+v)_{x_n}(\xi).
$$
Consequently,
$$\min\left\{\frac{|\nabla u_{x_n}|}{u_{x_n}},\frac{|\nabla v_{x_n}|}{v_{x_n}}\right\}(\xi)
\le 2\frac{|\nabla u_{x_n}|+|\nabla v_{x_n}|}{u_{x_n}+v_{x_n}}(\xi)\le 4C_1C_2$$
and the conclusion follows from the arbitrariness of $\xi\in\Sigma_R$.
\end{proof}

\subsection{A bound from below for $u_{x_nx_n}/u_{x_n}$ and $v_{x_nx_n}/v_{x_n}$ on strips}
\label{SecBelow}

\begin{lemma}\label{claim2}
Let $(u,v)$ be a positive classical solution of system  \eqref{LEDsys}  which is bounded on finite strips. 
Then, for each $R>0$, we have
$$
\inf_{\Gamma_R}\frac{ u_{x_nx_n}}{u_{x_n}}>-\infty,\quad 
\inf_{\Gamma_R}\frac{ v_{x_nx_n}}{v_{x_n}}>-\infty.
$$
\end{lemma}

\begin{proof}
Set
$$\xi_1=\frac{-u_{x_nx_n}}{(1+x_n)u_{x_n}},\quad 
\xi_2=\frac{-v_{x_nx_n}}{(1+x_n)v_{x_n}}.$$
By Lemma~\ref{ineqL1L2} with $\varphi=(1+x_n)^{-1}$, hence $a=2$, $d=0$, we have
$$\mathcal{L}_1\xi_1:=\Delta  \xi_1+b_1\cdot\nabla \xi_1
\ge  2 \xi^{2}_{1}+pv^{p-1}\frac{ z_2}{z_1}(\xi_1-\xi_2),$$
with 
$$z_1=(1+x_n)u_{x_n},\quad z_2=(1+x_n)v_{x_n},\quad b_1=2\frac{\nabla z_1}{z_1}=2\left(\frac{ \nabla u_{x_n}}{u_{x_n}}+\frac{e_{n}}{1+x_n}\right).$$
Next consider the barrier function $h(x)=h(x_n)=\kappa[1+(R-x_n)^{-2}]$. Choosing $\kappa>0$ sufficiently large, we have
$$h_{x_nx_n}+2h_{x_n}=6\kappa(R-x_n)^{-4}+4\kappa(R-x_n)^{-3}\le \kappa^2[1+(R-x_n)^{-2}]^2=h^2.$$
Let now
\be{defMR}
M=M(R)= 
\sup_{\Gamma_R}\left(\min\left\{\frac{|\nabla u_{x_n}|}{u_{x_n}},\frac{|\nabla v_{x_n}|}{v_{x_n}}\right\}\right),
\ee
which is finite by Lemma~\ref{claim1}, and
$$\xi=\xi_1-M-h.$$
We have
\be{L1xi}
\mathcal{L}_1\xi=\mathcal{L}_1 \xi_1-\mathcal{L}_1h \ge 
2\xi_1^2+pv^{p-1}\frac{ z_2}{z_1}(\xi_1-\xi_2)-\mathcal{L}_1h.
\ee
Observe that, in the set $\tilde \Gamma_R:=\Gamma_R\cap\{\xi\ge 0\}$, we have $\xi_1\ge M+h>0$ hence, since $h_{x_n}\ge 0$,
$$\begin{aligned}
\mathcal{L}_1h
&=h_{x_nx_n}+2\left(\frac{\nabla u_{x_n}}{u_{x_n}}+\frac{e_{n}}{1+x_n}\right)\cdot h_{x_n}e_n
=h_{x_nx_n}+2\left(\frac{u_{x_nx_n}}{u_{x_n}}+\frac{1}{1+x_n}\right)h_{x_n}\\
&=h_{x_nx_n}+2\left(-(1+x_n)\xi_1+\frac{1}{1+x_n}\right)h_{x_n}
\le h_{x_nx_n}+2h_{x_n} \le h^2
\end{aligned}$$
and, on the other hand,
$$\frac{|\nabla u_{x_n}|}{u_{x_n}}\ge-\frac{u_{x_nx_n}}{u_{x_n}}=(1+x_n)\xi_1>M,$$
hence, by the definition \eqref{defMR} of $M$,
$$\xi_2=\frac{-v_{x_nx_n}}{(1+x_n) v_{x_n}}
\le\frac{|\nabla v_{x_n}|}{(1+x_n)v_{x_n}}\le\frac{M}{1+x_n}\le\xi_1,$$
as well as  
$2\xi_1^2 - h^2\ge \xi_1^2\ge \xi^2$.
This along with \eqref{L1xi} implies
\be{L1xi2}
\mathcal{L}_1\xi\ge 2\xi_1^2- h^2\ge\xi^2 \quad\hbox{in $\tilde \Gamma_R$.}
\ee
Moreover, on $\{x_n=0\}$, since $u_{x_ix_i}=u=0$ for $i\in\{1,\dots,n-1\}$, we have $\xi_1=0$ hence 
\be{BC0}
\xi\le 0\quad\hbox{on $\{x_n=0\}$}.
\ee
Also, since $\xi_1\in C(\overline{\mathbb{R}^n_+})$, we have 
\be{BCR}
\xi(x',x_n)\to -\infty\ \hbox{ as $x_n\to R_-$,\quad  for each $x'\in\R^{n-1}$}.
\ee

In view of the above, we claim that $\xi_+$ is a weak solution of 
\be{ximinusweak}
A^{-1}\nabla\cdot(A\nabla \xi_+)\ge (\xi_+)^2\ \ \mbox{in}\  \Gamma_R,
\qquad\hbox{$\xi_+\le 0$ \ on $\partial \Gamma_R$,}
\ee
with $A=z_1^2$. 
This is essentially a consequence of Kato's inequality. However, 
in view of the unboundedness of $\xi$ near $x_n=R$ and so as to properly ensure the 
(weak formulation of the) boundary conditions,
we give details to make everything safe.
Thus take any $G\in C^2(\mathbb{R})$ such that 
\be{propG}
\hbox{$G>0$, $G'\ge 0$, $G''\ge 0$ on $(0,\infty)$ and $G=0$ on $(-\infty,0]$.}
\ee
Using \eqref{L1xi2}, we compute
$$\begin{aligned} 
z_1^{-2}\nabla\cdot(z_1^2\nabla(G(\xi))
&=\Delta (G(\xi))+b_1\cdot\nabla (G(\xi))\\
&=G'(\xi)\bigl[\Delta\xi+b_1\cdot\nabla\xi\bigr]+G''(\xi)|\nabla\xi|^2
\ge G'(\xi)\mathcal{L}_1\xi\ge G'(\xi)\xi^2,\quad x\in \Gamma_R.
\end{aligned}$$
Also, owing to \eqref{BCR}, for each $x'\in\R^{n-1}$, we have $G(\xi(x',x_n))=0$ as $x_n\to R_-$, 
hence $G\circ\xi\in C^2(\overline{\Gamma_R})$ and $\partial_\nu (G\circ\xi)=0$ on $\{x_n=R\}$.
Moreover, we have $\partial_\nu  (G\circ\xi)=G'(\xi)\partial_\nu\xi\le 0$ on $\{x_n=0\}$ in view of $G\circ\xi\ge 0$ and \eqref{BC0}.
Let $\varphi\in H^1(\overline\Gamma_R)$ be such that $\varphi\ge 0$ and 
$Supp(\varphi)$ is a compact subset of $\overline{\Gamma_R}$. Multiplying by $z_1^2\varphi$ and 
integrating by parts, we obtain
\be{LGxi}
\int_{\Gamma_R}z_1^2 G'(\xi)\xi^2\varphi\le 
-\int_{\Gamma_R} z_1^2 G'(\xi)\nabla\xi\cdot  \nabla\varphi
+\int_{\partial\Gamma_R} z_1^2 \partial_\nu(G\circ\xi)\varphi \,d\sigma\le 
-\int_{\Gamma_R} z_1^2 G'(\xi)\nabla\xi\cdot  \nabla\varphi.
\ee
Let $G_j$ be a sequence of $C^2$ functions with properties \eqref{propG} 
and such that $G_j\to s_+$, $G'_j\to \chi_{(0,\infty)}$ pointwise as $j\to\infty$
and $\sup_j \|G'_j\|_\infty<\infty$.
Applying \eqref{LGxi} with $G=G_j$ and passing to the limit $j\to\infty$ by dominated convergence,
it follows that $\xi_+$ is a weak solution of \eqref{ximinusweak}.

We may then apply Lemma~\ref{the main lemma of DSS} to deduce that $\xi\le 0$.
(Indeed, since $(u,v)$ is bounded on finite strips and is a solution of \eqref{LEDsys}, elliptic estimates guarantee that
$(u_{x_n},v_{x_n})$ is also bounded on finite strips, hence $z_1$ is bounded on $\Gamma_R$.)
Consequently $\sup_{\Gamma_{R/2}} \xi_1<\infty$.
Since this is true for any $R>1$, this and the analogous argument for $\xi_2$ provides the desired conclusion.
\end{proof}

\subsection{Proof of Proposition \ref{prop3}}
\label{secHar2}

We may assume $p\ge q$ without loss of generality.
\smallskip

{\bf Step 1.} {\it Notation and first estimate.}
Fix $R>1$. For given $a\in \partial\R^n_+$ we set, for all $k>0$,
$$\mathcal{B}_k=B_{kR}(a),\quad \mathcal{B}_k^+=\mathcal{B}_k\cap\Rn_+,
\quad b=a+Re_n.$$
In the rest of this proof, $c,C$ will denote generic positive constants depending on $R$ 
(and on $p, q, n$ and on the solution) but independent of $a$.
We shall use without further reference the fact that, owing to our assumption and elliptic estimates,
$$u,v,u_{x_n},v_{x_n}\le C\quad\hbox{in $\Gamma_R$.}$$
Set
$$\gamma=\frac{q+1}{p+1}\le 1,\quad \beta=\frac{p(q+1)}{p+1}\ge 1.$$
By Lemma \ref{claim2}, we have
$$u_{x_nx_n}+Cu_{x_n}\ge 0,\quad x\in \Gamma_R$$
i.e. $(u_{x_n}e^{Cx_n})_{x_n}\ge 0$, hence
\be{compunormal}
u_{x_n}(x',x_n)\ge cu_{x_n}(x',0),\quad x\in \Gamma_R.
\ee

\noindent {\bf Step 2.} {\it Control of $u, v$ by boundary values of $u_{x_n}$.}
We claim that
\be{compuvugamma}
\sup_{\mathcal{B}_4^+}\frac{u+v}{x_n}\le Cu^\gamma_{x_n}(a).
\ee

By Proposition~\ref{comparison property}, we have $v^p\le Cu^\beta$, hence $-\Delta u=v^p=d(x)u$ 
with $\|d\|_{L^\infty(\Gamma_{2R})}\le C$. 
Since also $u=0$ on $\partial\Rnp$, it follows from the boundary Harnack inequality 
(see \cite{Sir,RSS} and the references therein) that
$$\sup_{\mathcal{B}_4^+}\frac{u}{x_n}\le C\inf_{\mathcal{B}_4^+}\frac{u}{x_n},$$
hence
\be{compuxna}
\sup_{\mathcal{B}_4^+}\frac{u}{x_n}\le Cu_{x_n}(a).
\ee
On the other hand, writing
$$-\Delta (u+v)=\frac{v^{p}+u^{q}}{u+v}(u+v)=\hat d(x)(u+v)$$
with $\|\hat d\|_{L^\infty(\Gamma_{5R})}\le C$, 
by the boundary Harnack inequality and 
Proposition~\ref{comparison property}, we obtain
$$\sup_{\mathcal{B}_4^+}\frac{u+v}{x_n}\le C\inf_{\mathcal{B}_4^+}\frac{u+v}{x_n}
\le C\inf_{\mathcal{B}_4^+}\frac{u+u^\gamma}{x_n}\le C\inf_{\mathcal{B}_4^+}\frac{u^\gamma}{x_n}\le Cu^\gamma(b),$$
where we used $\gamma\le 1$. This combined with \eqref{compuxna} yields \eqref{compuvugamma}.

\smallskip

\noindent {\bf Step 3.} {\it Log-grad estimate for $u_{x_n}$.}
We claim that
\be{loggrad1}
|\nabla u_{x_n}|\le Cu_{x_n},\quad x\in\Gamma_R,
\ee
which in particular implies that $u_{x_n}$ satisfies Harnack's inequality in $\Gamma_R$:
\be{loggrad1b}
\sup_{\Gamma_R}u_{x_n}\le C\inf_{\Gamma_R}u_{x_n}.
\ee

Using the second equation in  \eqref{LEDsys} and the boundary conditions, along with elliptic interior-boundary estimates, we get
$$\|v\|_{W^{2,r}(\mathcal{B}_3^+)}\le C\|v\|_{L^\infty(\mathcal{B}_4^+)}+C\|u\|_{L^\infty(\mathcal{B}_4^+)}^q$$
for any finite $r>1$. Consequently, taking $r>n$ and using Morrey's imbedding, \eqref{compuvugamma}
and $q\ge 1$, we obtain
\be{compuvugamma2}
\|v\|_{L^\infty(\mathcal{B}_3^+)}+\|v_{x_n}\|_{L^\infty(\mathcal{B}_3^+)}\le  Cu^\gamma_{x_n}(a).
\ee
Now recall that, extending $u, v$ by odd reflection (cf.~\eqref{oddrefl}),
the (sign-changing) functions $u, v$ satisfy $-\Delta u=|v|^{p-1}v$,
$-\Delta v=|u|^{q-1}u$ in $\Rn$
and $u_{x_n}, v_{x_n}>0$ satisfy
\be{equxn}
-\Delta u_{x_n}=p|v|^{p-1}v_{x_n},\quad
-\Delta v_{x_n}=q|u|^{q-1}u_{x_n}
\quad\hbox{in $\Rn$.}
\ee
By the interior Harnack inequality with RHS (cf.~\cite{Tr}), making use of \eqref{compuvugamma2}, 
$\sup_{\Gamma_{4R}}u_{x_n} \le C$,
$\gamma p=\beta\ge 1$
and then \eqref{compunormal}, we deduce
$$\sup_{\mathcal{B}_2^+}u_{x_n}\le C\inf_{\mathcal{B}_2^+}u_{x_n}+C\||v|^{p-1}v_{x_n}\|_{L^\infty(\mathcal{B}_3^+)}
\le C\inf_{\mathcal{B}_2^+}u_{x_n}+Cu^{\gamma p}_{x_n}(a)\le Cu_{x_n}(a)\le Cu_{x_n}(y),\quad y\in [a,b],$$
where $[a,b]$ denotes the line segment of endpoints $a, b$.
Going back to the first equation in \eqref{equxn} and using elliptic estimates, we deduce
$$\|u_{x_n}\|_{W^{2,r}(\mathcal{B}_1^+)}\le C\|u_{x_n}\|_{L^\infty(\mathcal{B}_2^+)}+C\||v|^{p-1}v_{x_n}\|_{L^\infty(\mathcal{B}_2^+)}
\le Cu_{x_n}(y),\quad y\in [a,b],$$
hence \eqref{loggrad1}.

\smallskip

\noindent {\bf Step 4.} {\it Log-grad estimate for $v_{x_n}$.}
We claim that
\be{loggrad2}
|\nabla v_{x_n}|\le Cv_{x_n},\quad x\in\Gamma_R.
\ee

Note that
$$u(x',x_n)=\int_0^{x_n} u_{x_n}(x',s)\,ds\ge \inf_{\mathcal{B}_3^+}u_{x_n},\quad x\in \mathcal{B}_3^+\cap\{x_n>1\}.$$
The representation formula using Green function, applied on the second equation in \eqref{equxn}, and  the Green kernel standard property then give
\be{loggrad2a}
\inf_{\mathcal{B}_4^+} v_{x_n}=\inf_{\mathcal{B}_4} v_{x_n}\ge c\int_{\mathcal{B}_3^+}|u|^{q-1}u_{x_n}dy\ge 
c\bigl|\mathcal{B}_3^+\cap\{x_n>1\}\bigr| \inf_{\mathcal{B}_3^+}u^q_{x_n}\ge c \inf_{\mathcal{B}_3^+}u^q_{x_n}.
\ee
By the interior Harnack inequality with RHS applied to the second equation in  \eqref{LEDsys},
inequality \eqref{loggrad1b} (applied with $3R$) and \eqref{loggrad2a}, we obtain
$$\sup_{\mathcal{B}_2^+}v_{x_n}\le C\inf_{\mathcal{B}_2^+}v_{x_n}+C\||u|^{q-1}u_{x_n}\|_{L^\infty(\mathcal{B}_3^+)}
\le C\inf_{{\mathcal{B}_2^+}}v_{x_n}+C\sup_{\mathcal{B}_3^+}u^q_{x_n}
\le C\inf_{{\mathcal{B}_2^+}}v_{x_n}+C\inf_{\mathcal{B}_3^+}u^q_{x_n} \le C\inf_{{\mathcal{B}_2^+}}v_{x_n}.$$
Going back to the second equation in \eqref{equxn} and using elliptic estimates, we deduce
$$\|v_{x_n}\|_{W^{2,r}(\mathcal{B}_1^+)}\le C\|v_{x_n}\|_{L^\infty(\mathcal{B}_2^+)}+C\||u|^{q-1}u_{x_n}\|_{L^\infty(\mathcal{B}_2^+)}
\le C\inf_{{\mathcal{B}_2^+}}v_{x_n},$$
hence \eqref{loggrad2}. The proof is complete.

\section{Proof of Propositions \ref{prop2} and \ref{prop5} and of Theorem \ref{Theorem 2}}

\begin{proof}[Proof of Proposition \ref{prop2}] 
This follows from moving planes arguments.
See \cite{deFS,PhSo} for detailed proofs in the case of cooperative systems (cf.~also \cite{BiMi}).
It can be checked (see also \cite{Farina20} for the scalar case)
that no global boundedness assumption is required
and that the reflection arguments in these proofs can be carried out in each finite strip 
owing to the boundedness of $u, v$ on finite strips.
\end{proof}

\begin{proof}[Proof of Proposition \ref{prop5}]
Let $(u,v)$ be a nonnegative classical solution of \eqref{LEDsys}.
It is well known (see,~e.g., \cite[p.339]{QS19}) that $(u,v)$ satisfies the integral a priori estimate
\be{intAE}
\int_{B_1(a)} v^p \,dx\le C(n,p,q)\quad\hbox{ for any $a\in\Rn$ with $a_n>2$.}
\ee
On the other hand, the conditions $v_{x_n}>0$, $v_{x_nx_n}\ge 0$ imply $\lim_{x_n\to\infty}v(x',x_n)=\infty$ for each $x'\in\R^n$.
By monotone convergence, we deduce that $\lim_{R\to\infty}\int_{B_1(Re_n)} v^p=\infty$,
which contradicts \eqref{intAE}.
\end{proof}

\begin{proof}[Proof of Theorem \ref{Theorem 2}]
 The result immediately follows by combining Propositions \ref{prop2}-\ref{prop5}.
\end{proof}

\medskip
{\bf Acknowledgements:} Yimei Li is supported by the National Natural Science Foundation of China (12101038) and the China Scholarship Council (No. 202307090023).

\medskip

{\bf Conflicts of interests:} none.


\medskip


\begin{thebibliography}{99}

 \bibitem{AS}
 S. Armstrong, B. Sirakov, 
 Nonexistence of positive supersolutions of elliptic equations via the maximum principle,
  \emph{Commun. Partial Differ. Eq.} 36 (2011), 2011--2047.

\bibitem{BCN1}
H. Berestycki, L. Caffarelli, L. Nirenberg, 
Monotonicity for elliptic equations in an unbounded Lipschitz domain, 
\emph{Comm. Pure Appl. Math.} 50 (1997) 1089--1112.

\bibitem{BCN2}
H. Berestycki, L. Caffarelli, L. Nirenberg, 
Further qualitative properties for elliptic equations in unbounded domains, 
\emph{Ann. Sc. Norm. Super. Pisa Cl. Sci.} (4) 25 (1997) 69--94.


\bibitem{BV}
M. F. Bidaut-V\'eron, L. V\'eron, 
Nonlinear elliptic equations on compact Riemannian manifolds and asymptotics of Emden equations,
\emph{Invent. Math.} 106 (1991), 489--539.

\bibitem{BVY}
M.-F. Bidaut-V\'eron, C. Yarur, 
Semilinear elliptic equations and systems with measure data: Existence and a priori estimates, 
\emph{Adv. Differential Equations} 7 (2002), 257--296.

\bibitem{BiMi}
I. Birindelli, E. Mitidieri, 
Liouville theorems for elliptic inequalities and applications, 
\emph{Proc. Roy. Soc. Edinburgh Sect. A} 128 (1998), 1217--1247.

\bibitem{BM}
J. Busca, R. Man\'asevich, 
A Liouville-type theorem for Lane-Emden system,
\emph{Indiana Univ. Math. J.} 51 (2002), 37--51.

 \bibitem{CW}
H. Chan, J. Wei, 
On De Giorgi's conjecture: recent progress and open problems,
 \emph{Sci. China Math.} 61 (2018), 925--1946.


\bibitem{CL}
W. Chen,  C. Li, 
Classification of solutions of some nonlinear elliptic equations,
\emph{Duke Math. J.} 63 (1991), 615--622.

\bibitem{CP}
X.-Y. Chen, P. Pol\'a\v{c}ik, 
Asymptotic periodicity of positive solutions of reaction diffusion equations on a ball,
\emph{J. Reine Angew. Math.} 472 (1996), 17--51.

\bibitem{CLZ}
Z. Chen, C.-S. Lin, W. Zou, 
Monotonicity and nonexistence results to cooperative systems in the half space,
\emph{J. Funct. Anal.} 266 (2014), 1088--1105.

\bibitem{CFM}
Ph. Cl\'ement, D.G. de Figueiredo, E. Mitidieri, 
Positive solutions of semilinear elliptic systems, 
\emph{Comm. Partial Differential Equations} 17 (1992), 923--940.

\bibitem{Cowan}
C. Cowan,
Liouville theorems for stable Lane-Emden systems and biharmonic problems,
\emph{Nonlinearity} 26 (2013), 2357--2371.

\bibitem{Dancer}
N.  Dancer, 
Some notes on the method of moving planes,
\emph{Bull. Austral. Math. Soc.}  46 (1992), 425--434.

\bibitem{deF}
D.G. de Figueiredo, P. L. Felmer, 
A Liouville-type theorem for elliptic systems,
\emph{Ann. Sc. Norm. Super. Pisa, Cl. Sci.} 21 (1994), 387--397.

\bibitem{deFS}
D.G. de Figueiredo, B. Sirakov,
Liouville type theorems, monotonicity results and a priori bounds for positive solutions of elliptic systems,
\emph{Math. Ann.} 333 (2005), 231--260.

\bibitem{DFT}
L. Dupaigne, A. Farina, T. Petitt,
Liouville-type theorems for the Lane-Emden equation in the half-space and cones,
\emph{J. Funct. Anal.} 284 (2023), 109906.

\bibitem{DSS}
L. Dupaigne, B. Sirakov, Ph. Souplet, 
A Liouville-type theorem for the Lane-Emden equation in a half-space,
\emph{Int. Math. Res. Not.} 2022, 12,  (2022), 9024--9043.

\bibitem{Farina20}
A. Farina, 
Some results about semilinear elliptic problems on half-spaces,
\emph{Math. Eng.} 2, (2020) 709--721.

\bibitem{Farina}
A. Farina, 
On the classification of solutions of the Lane-Emden equation on unbounded domains of $\mathbb{R}^{n}$,
\emph{J. Math. Pures Appl.} 87 (2007), 537--561.

\bibitem{GS}
B. Gidas,  J. Spruck,
Global and local behavior of positive solutions of nonlinear elliptic equations,
\emph{Comm. Pure Appl. Math.} 34 (1981), 525--598.

\bibitem{GS2}
B. Gidas,  J. Spruck,
A priori bounds for positive solutions of nonlinear elliptic equations,
\emph{Comm. Partial Differential Equations} 6 (1981), 883--901.

\bibitem{HLSWW}
F.  Hamel, Y. Liu, P. Sicbaldi, K. Wang, J. Wei,
Half-space theorems for the Allen-Cahn equation and related problems,
\emph{J. Reine Angew. Math.} 770 (2021), 113--133.

\bibitem{LZ}
Y. Li, L. Zhang, 
Liouville-type theorems and Harnack-type inequalities for semilinear elliptic equations, 
\emph{J. Anal. Math.} 90 (2003), 27--87.


\bibitem{Lin}
C.-S. Lin,
A classification of solutions of a conformally invariant fourth order equation in $\mathbb{R}^{n}$,
\emph{Comment. Math. Helv.} 73  (1998), 206--231.

\bibitem{Mit93}
E. Mitidieri, 
A Rellich type identity and applications, 
\emph{Comm. Partial Differential Equations} 18 (1993), 125--151.

\bibitem{Mit}
E. Mitidieri, 
Nonexistence of positive solutions of semilinear elliptic systems in $\mathbb{R}^{n}$,
\emph{Differential Integral Equations}, 9  (1996), 465--479.


\bibitem{MP}
E. Mitidieri, S.I. Pohozaev, 
A priori estimates and blow-up of solutions of nonlinear partial differential equations and inequalities,
\emph{Proc. Steklov Inst. Math.} 234 (2001), 1--362.

\bibitem{MSS}
A. Montaru, B. Sirakov,  Ph. Souplet,
Proportionality of components, Liouville theorems and a priori estimates for noncooperative elliptic systems,
\emph{Arch. Ration. Mech. Anal.} 213  (2014), 129--169.


\bibitem{PVV}
L.A. Peletier, R. van der Vorst, 
Existence and nonexistence of positive solutions of nonlinear elliptic systems and the biharmonic equation, 
\emph{Differential Integral Equations} 5 (1992), 747--767.

\bibitem{PhSo}
Q.H. Phan, Ph. Souplet, 
A Liouville-type theorem for the $3$-dimensional parabolic Gross-Pitaevskii and related systems,
\emph{Math. Ann.} 366 (2016), 1561--1585.

\bibitem{PQS}
P. Pol\'a\v{c}ik, P. Quittner, Ph. Souplet,  
Singularity and decay estimates in superlinear problems via Liouville-type theorems. Part I: elliptic equations and systems,
\emph{Duke Math. J.} 139  (2007), 555--579.

\bibitem{QS19}
P. Quittner,  Ph. Souplet,
Superlinear parabolic problems. Blow-up, global existence and steady states,
Second Edition. Birkh\"auser Advanced Texts, 2019.

\bibitem{RZ}
W.~Reichel, H.~Zou,
Non-existence results for semilinear cooperative elliptic systems via moving spheres,
\emph{J. Differ. Equations}, 161 (2000), 219--243.


\bibitem{RSS}
F. Rend\'on, B. Sirakov,  M. Soares,
Boundary weak Harnack estimates and regularity for elliptic PDE in divergence form,
\emph{Nonlinear Anal.} 235, Article ID 113331, 13 p. (2023). 


\bibitem{SZ96} 
J. Serrin, H. Zou, 
Non-existence of positive solutions of Lane-Emden systems,
\emph{Differential Integral Equations} 9 (1996), 635--653.


\bibitem{SZ98}
J. Serrin, H. Zou, 
Existence of positive solutions of the Lane-Emden system,
\emph{Atti Sem. Mat. Fis. Univ. Modena} 46 (1998, suppl.), 369--380.


\bibitem{Sir} 
B. Sirakov, 
Boundary Harnack estimates and quantitative strong maximum principles for uniformly elliptic PDE, 
\emph{Int. Math. Res. Notices} 24 (2018), 7457--7482.

\bibitem{So09}
Ph. Souplet,
The proof of the Lane-Emden conjecture in four space dimensions,
\emph{Adv. Math.} 221  (2009), 1409--1427.

\bibitem{Tr}
N. Trudinger, 
Maximum principles for linear, non-uniformly elliptic operators with measurable coefficients, 
\emph{Math. Z.} 156 (1977) 291--301.
\end{thebibliography}
\end{document}